\documentclass{amsart}

\usepackage{amssymb,amsmath,amscd,graphicx,amsthm}

\newtheorem{theorem}{Theorem}
\newtheorem{lemma}[theorem]{Lemma}
\newtheorem{proposition}[theorem]{Proposition}

\newtheorem{conjecture}[theorem]{Conjecture}

\theoremstyle{definition}

\theoremstyle{remark}
\newtheorem{remark}[theorem]{Remark}

\numberwithin{equation}{section}
\numberwithin{theorem}{section}

\def\A{{\mathcal A}}
\def\AA{{\mathbb A}}

\def\C{{\mathbb C}}
\def\CC{{\mathcal C}}

\def\FF{{\mathcal F}}
\def\FFF{{\mathcal F}}
\def\G{{\mathcal G}}
\def\H{\mathcal H}

\def\M{{\mathcal M}}
\def\N{\mathcal N}
\def\O{{\mathcal O}}
\def\P{{\mathcal P}}
\def\PP{{\mathcal P}}

\def\TE{{\mathcal T}}

\def\UU{\overline{\A}}

\def\Z{{\mathbb Z}}

\def\d{\mathbf d}

\def\g{\mathfrak g}

\def\h{\mathfrak h}
\def\ii{{\bf i}}

\def\wB{{\widetilde{B}}}

\def\wx{{\widetilde{\bf x}}}
\def\x{{\bf x}}

\def\Ad{\operatorname{Ad}}

\def\Id{{\operatorname {Id}}}

\def\Poi{{\{\cdot,\cdot\}}}

\def\diag{\operatorname{diag}}
\def\dim{\operatorname{dim}}

\def\rank{\operatorname{rank}}

\def\sign{{\operatorname{sign}}}

\def\:{{:\ }}

\begin{document}

\title[ Cluster structures and Belavin-Drinfeld classification]
{Cluster structures on simple complex Lie groups and Belavin-Drinfeld classification}

\author{M. Gekhtman}

\address{Department of Mathematics, University of Notre Dame, Notre Dame,
IN 46556}
\email{mgekhtma@nd.edu}

\author{M. Shapiro}
\address{Department of Mathematics, Michigan State University, East Lansing,
MI 48823}
\email{mshapiro@math.msu.edu}

\author{A. Vainshtein}
\address{Department of Mathematics \& Department of Computer Science, University of Haifa, Haifa,
Mount Carmel 31905, Israel}
\email{alek@cs.haifa.ac.il}

\begin{abstract} 
We study natural cluster structures in the rings of regular functions on simple complex Lie groups and Poisson-Lie
structures compatible with these cluster structures.
According to our main conjecture, each class in the Belavin-Drinfeld classification of Poisson-Lie structures on $\G$ 
corresponds to a cluster structure in $\O(\G)$. We prove a reduction theorem explaining how different parts of the conjecture are 
related to each other. The conjecture is established for $SL_n$, $n<5$,  and for any $\G$ in the case of the standard Poisson-Lie structure.
\end{abstract}

\subjclass[2000]{53D17, 13F60}
\keywords{Poisson-Lie group,  cluster algebra, Belavin-Drinfeld triple}

\maketitle

\section{Introduction}

Since the invention of cluster algebras in 2001, a large part of
research in the field has been devoted to uncovering cluster 
structures in rings of regular functions on various algebraic varieties
arising in algebraic geometry, representation theory, and
mathematical physics. Once the existence of such a structure was
established, abstract features of cluster algebras were used to
study geometric properties of underlying objects.  Research in this
direction led to many exciting results \cite{SSVZ, FoGo1, FoGo2}. It also created
an impression that, given an algebraic variety, there is a unique (if at all) natural
cluster structure associated with it.

The main goal of the current paper is to establish the following phenomenon:
in certain situations,  the same ring may have \emph{multiple\/} natural cluster structures. 
More exactly, we engage into a systematic study of
multiple cluster structures in the rings of regular functions
on simple Lie groups (in what follows we will shorten that to {\it
cluster structures on simple Lie groups\/}).
Consistent with the philosophy advocated in \cite{GSV1,GSV2, GSV3, GSV4, GSV5, GSVb}, we will focus 
on compatible Poisson structures on the Lie groups,
that is, on compatible Poisson-Lie structures.

The notion of a Poisson bracket compatible with a cluster structure was introduced in
\cite{GSV1}. It was used there to interpret cluster transformations and matrix mutations from
a viewpoint of Poisson geometry. In addition, it was shown that if a 
Poisson algebraic variety $\left ( \mathcal{M}, \Poi\right )$ possesses a  coordinate chart
that consists of regular functions whose logarithms have pairwise constant Poisson brackets, 
then one can use this chart to define a cluster structure $\CC_\M$ compatible with $\Poi$. Algebraic 
structures corresponding to $\CC_\M$ (the cluster algebra and the upper cluster algebra)
are closely related 
to the ring $\O(\M)$ of regular functions on $\mathcal{M}$. 
More precisely, under certain rather mild conditions, $\O(\M)$ can be obtained by tensoring one of these
algebras by $\C$.

This construction was applied in \cite[Ch.~4.3]{GSVb} to double Bruhat cells in semisimple Lie groups
equipped with (the restriction of) the {\em standard\/} Poisson-Lie structure. It was shown that
the resulting cluster structure coincides with the one built in \cite{CAIII}. Recall that it was proved in
\cite{CAIII} that the corresponding upper cluster algebra coincides with the ring of regular functions on the
double Bruhat cell. Since the open double Bruhat cell is dense in the corresponding Lie group, 
the corresponding fields of rational functions coincide, thus allowing to equip  the field of rational functions on the  
Lie group with the same cluster structure. 
Moreover, we show below that the upper cluster algebra coincides with the ring of regular functions on the Lie group.

The standard Poisson-Lie structure is a particular case of Poisson-Lie structures corresponding to quasi-triangular
Lie bialgebras. Such structures are associated with solutions to the classical Yang-Baxter equation (CYBE).
Their complete classification was obtained by Belavin and Drinfeld in \cite{BD}. We conjecture that any such solution
 gives rise to a compatible cluster structure on the Lie group, and that the properties of this structure are similar 
 to those mentioned above. The detailed formulation of our conjectures requires some preliminary work; it is given in Section~\ref{SecMC} below. In Section~\ref{reduction} we study interrelations between the different parts of the conjecture.
 Currently, we have several examples supporting our conjecture: it holds for the class of the standard Poisson-Lie
  structure in any simple complex Lie group, and for 
 the whole Belavin-Drinfeld classification in $SL_n$ for $n=2,3,4$.  These results are described in Sections~\ref{Secgen} and~\ref{Sec34}, respectively. 
 In Section~\ref{SecTri} we
 discuss the case of Poisson-Lie structures beyond those associated with solutions to CYBE.

\section{Cluster structures and compatible Poisson brackets}
\label{SecPrel}

\subsection{}
We start with the basics on cluster algebras of geometric type. The definition that we present
below is not the most general one, see, e.g.,
\cite{FZ2, CAIII} for a detailed exposition. In what follows, we will use a notation $[i,j]$ for an interval
$\{i, i+1, \ldots , j\}$ in $\mathbb{N}$ and we will denote $[1, n]$ by $[n]$.
 
The {\em coefficient group\/} $\PP$ is a free multiplicative abelian
group of finite rank $m$ with generators $g_1,\dots, g_m$.
An {\em ambient field\/}  is
the field $\FFF$ of rational functions in $n$ independent variables with
coefficients in the field of fractions of the integer group ring
$\Z\PP=\Z[g_1^{\pm1},\dots,g_m^{\pm1}]$ (here we write
$x^{\pm1}$ instead of $x,x^{-1}$).

A {\em seed\/} (of {\em geometric type\/}) in $\FFF$ is a pair
$\Sigma=(\x,\widetilde{B})$,
where $\x=(x_1,\dots,x_n)$ is a transcendence basis of $\FFF$ over the field of
fractions of $\Z\PP$ and $\widetilde{B}$ is an $n\times(n+m)$ integer matrix
whose principal part $B=\wB([n],[n])$ is skew-symmetrizable (here and in what follows, we denote by $A(I,J)$ a submatrix of a matrix $A$ with a row set $I$ and a column set $J$). Matrices $B$ and $\wB$ are called the
{\it exchange matrix\/} and the {\it extended exchange matrix}, respectively.
In this paper, we will only deal with the case when the exchange matrix is skew-symmetric.

The $n$-tuple  $\x$ is called a {\em cluster\/}, and its elements
$x_1,\dots,x_n$ are called {\em cluster variables\/}. Denote
$x_{n+i}=g_i$ for $i\in [m]$. We say that
$\widetilde{\x}=(x_1,\dots,x_{n+m})$ is an {\em extended
cluster\/}, and $x_{n+1},\dots,x_{n+m}$ are {\em stable
variables\/}. It is convenient to think of $\FFF$ as
of the field of rational functions in $n+m$ independent variables
with rational coefficients. 

Given a seed as above, the {\em adjacent cluster\/} in direction $k\in [n]$
is defined by
$$
\x_k=(\x\setminus\{x_k\})\cup\{x'_k\},
$$
where the new cluster variable $x'_k$ is given by the {\em exchange relation}
\begin{equation}\label{exchange}
x_kx'_k=\prod_{\substack{1\le i\le n+m\\  b_{ki}>0}}x_i^{b_{ki}}+
       \prod_{\substack{1\le i\le n+m\\  b_{ki}<0}}x_i^{-b_{ki}};
\end{equation}
here, as usual, the product over the empty set is assumed to be
equal to~$1$.

We say that $\wB'$ is
obtained from $\wB$ by a {\em matrix mutation\/} in direction $k$
and
write $\wB'=\mu_k(\wB)$ 
 if
\[
b'_{ij}=\begin{cases}
         -b_{ij}, & \text{if $i=k$ or $j=k$;}\\
                 b_{ij}+\displaystyle\frac{|b_{ik}|b_{kj}+b_{ik}|b_{kj}|}2,
                                                  &\text{otherwise.}
        \end{cases}
\]
It can be easily verified that $\mu_k(\mu_k(\wB))=\wB$.

Given a seed $\Sigma=(\x,\widetilde{B})$, we say that a seed
$\Sigma'=(\x',\widetilde{B}')$ is {\em adjacent\/} to $\Sigma$ (in direction
$k$) if $\x'$ is adjacent to $\x$ in direction $k$ and $\widetilde{B}'=
\mu_k(\widetilde{B})$. Two seeds are {\em mutation equivalent\/} if they can
be connected by a sequence of pairwise adjacent seeds. 
The set of all seeds mutation equivalent to $\Sigma$ is called the {\it cluster structure\/} 
(of geometric type) in $\FFF$ associated with $\Sigma$ and denoted by $\CC(\Sigma)$; in what follows, 
we usually write $\CC(\wB)$, or even just $\CC$ instead. 

Following \cite{FZ2, CAIII}, we associate
with $\CC(\wB)$ two algebras of rank $n$ over the {\it ground ring\/} $\AA$, $\Z\subseteq\AA \subseteq\Z\P$:
the {\em cluster algebra\/} $\A=\A(\CC)=\A(\wB)$, which 
is the $\AA$-subalgebra of $\FF$ generated by all cluster
variables in all seeds in $\CC(\wB)$, and the {\it upper cluster algebra\/}
$\UU=\UU(\CC)=\UU(\wB)$, which is the intersection of the rings of Laurent polynomials over $\AA$ in cluster variables
taken over all seeds in $\CC(\wB)$. The famous {\it Laurent phenomenon\/} \cite{FZ3}
claims the inclusion $\A(\CC)\subseteq\UU(\CC)$. The natural choice of the ground ring for the geometric type
is the polynomial ring in stable variables $\AA=\Z\P_+=\Z[x_{n+1},\dots,x_{n+m}]$; this choice is assumed unless
explicitly stated otherwise. 

Let $V$ be a quasi-affine variety over $\C$, $\C(V)$ be the field of rational functions on $V$, and
$\O(V)$ be the ring of regular functions on $V$. Let $\CC$ be a cluster structure in $\FF$ as above.
Assume that $\{f_1,\dots,f_{n+m}\}$ is a transcendence basis of $\C(V)$. Then the map $\varphi: x_i\mapsto f_i$,
$1\le i\le n+m$, can be extended to a field isomorphism $\varphi: \FF_\C\to \C(V)$,  
where $\FF_\C=\FF\otimes\C$ is obtained from $\FF$ by extension of scalars.
The pair $(\CC,\varphi)$ is called a cluster structure {\it in\/}
$\C(V)$ (or just a cluster structure {\it on\/} $V$), $\{f_1,\dots,f_{n+m}\}$ is called an extended cluster in
 $(\CC,\varphi)$.
Sometimes we omit direct indication of $\varphi$ and say that $\CC$ is a cluster structure on $V$. 
A cluster structure $(\CC,\varphi)$ is called {\it regular\/}
if $\varphi(x)$ is a regular function for any cluster variable $x$. 
The two algebras defined above have their counterparts in $\FF_\C$ obtained by extension of scalars; they are
denoted $\A_\C$ and $\UU_\C$.
If, moreover, the field isomorphism $\varphi$ can be restricted to an isomorphism of 
$\A_\C$ (or $\UU_\C$) and $\O(V)$, we say that 
$\A_\C$ (or $\UU_\C$) is {\it naturally isomorphic\/} to $\O(V)$.

The following statement is a weaker analog of Proposition~3.37 in \cite{GSVb}.

\begin{proposition}\label{regfun}
Let $V$ be a Zariski open subset in $\C^{n+m}$ and $(\CC=\CC(\wB),\varphi)$ be a cluster structure in $\C(V)$ 
with $n$ cluster and $m$ stable variables such that

{\rm(i)} $\rank\wB=n$;

{\rm(ii)} there exists an extended cluster $\wx=(x_1,\dots,x_{n+m})$ in $\CC$ such that $\varphi(x_i)$ is
regular on $V$ for $i\in [n+m]$;

{\rm(iii)} for any cluster variable $x_k'$, $k\in [n]$, obtained via the exchange relation~\eqref{exchange} 
applied to $\wx$, $\varphi(x_k')$ is regular on $V$.

{\rm(iv)} for any stable variable $x_{n+i}$, $i\in [m]$, $\varphi(x_{n+i})$ vanishes at some point of $V$;

{\rm(v)} each regular function on $V$ belongs to $\varphi(\UU_\C(\CC))$.

Then $\UU_\C(\CC)$ is naturally isomorphic to $\O(V)$.
\end{proposition}

\subsection{}
Let $\Poi$ be a Poisson bracket on the ambient field $\FFF$, and $\CC$ be a cluster structure in $\FFF$. 
We say that the bracket and the cluster structure are {\em compatible\/} if, for any extended
cluster $\widetilde{\x}=(x_1,\dots,x_{n+m})$,  one has
\begin{equation}\label{cpt}
\{x_i,x_j\}=\omega_{ij} x_ix_j,
\end{equation}
where $\omega_{ij}\in\Z$ are
constants for all $i,j\in[n+m]$. The matrix
$\Omega^{\widetilde \x}=(\omega_{ij})$ is called the {\it coefficient matrix\/}
of $\Poi$ (in the basis $\widetilde \x$); clearly, $\Omega^{\widetilde \x}$ is
skew-symmetric.

A complete characterization of Poisson brackets compatible with a given cluster structure $\CC=\CC(\wB)$ in the case $\rank\wB=n$ is given in \cite{GSV1}, see also \cite[Ch.~4]{GSVb}. In particular, the following statement is an 
immediate corollary of Theorem~1.4  in \cite{GSV1}.

\begin{proposition}\label{Bomega}
Let $\rank \wB=n$, then
a Poisson bracket is compatible with $\CC(\wB)$ if and only if its coefficient matrix 
$\Omega^\wx$ satisfies $\wB\Omega^{\wx}=(D\; 0)$, where $D$ is a diagonal matrix.
\end{proposition}

Clearly, the notion of compatibility and the result stated above extend to Poisson brackets on $\FF_\C$ without any changes.
A different description of compatible Poisson brackets on $\FF_\C$ is based on the notion of a toric action.
 Fix an arbitrary extended cluster
$\wx=(x_1,\dots,x_{n+m})$ and define a {\it local toric action\/} of rank $r$ as the map 
$\TE^W_{\d}:\FF_\C\to
\FF_\C$ given on the generators of $\FF_\C=\C(x_1,\dots,x_{n+m})$ by the formula 
\begin{equation}
\TE^W_{\d}(\wx)=\left ( x_i \prod_{\alpha=1}^r d_\alpha^{w_{i\alpha}}\right )_{i=1}^{n+m},\qquad
\d=(d_1,\dots,d_r)\in (\C^*)^r,
\label{toricact}
\end{equation}
where $W=(w_{i\alpha})$ is an integer $(n+m)\times r$ {\it weight matrix\/} of full rank, and extended naturally to the whole $\FF_\C$. 

Let $\wx'$ be another extended cluster, then the corresponding local toric action defined by the weight matrix $W'$
is {\it compatible\/} with the local toric action \eqref{toricact} if the following diagram is commutative for
any fixed $\d\in (\C^*)^r$:
$$
\begin{CD}
\FF_\C=\C(\wx) @>>> \FF_\C=\C(\wx')\\
@V\TE^W_{\d} VV @VV \TE^{W'}_{\d}V\\
\FF_\C=\C(\wx) @>>> \FF_\C=\C(\wx')
\end{CD}
$$
(here the horizontal arrows are induced by $x_i\mapsto x'_i$ for $1\le i\le n+m$). If local toric actions at all clusters are compatible, they define a {\it global toric action\/} $\TE_{\d}$ on $\FF_\C$ called the extension of the local toric action \eqref{toricact}. Lemma~2.3 in \cite{GSV1} claims that \eqref{toricact} extends 
to a unique global action of $(\C^*)^r$  if and only if $\wB W = 0$. Therefore, if $\rank\wB=n$, then the maximal possible rank of a global toric action equals $m$. Any global toric action can be obtained from a toric action of
the maximal rank by setting some of $d_i$'s equal to~$1$.

A description of Poisson brackets on $\FF_\C$ compatible with a cluster structure $\CC=\CC(\wB)$ 
based on the notion of the global toric action is suggested in \cite{GSSV}.
Given a Poisson bracket $\Poi_0$ on $\FFF_\C$ compatible with $\CC$, one can obtain all other compatible 
brackets as follows.
   
Assume that $(\mathbb{C}^*)^m$ is equipped with a Poisson structure given by
\begin{equation}\label{torbra}
\{d_i, d_j\}_V = v_{ij} d_i d_j,
\end{equation}
where $V=(v_{ij})$ is a skew-symmetric matrix. 

\begin{proposition} For any $V$, there exists a  Poisson structure $\Poi_V^\CC$ compatible with  $\CC$ such 
that the map $\left ((\mathbb{C}^*)^m \times \FFF_\C , \Poi_V \times \Poi_0 \right )\to 
\left ( \FFF_\C, \Poi_V^\CC\right )$ extended from the action $(\d,\wx) \mapsto 
\TE_{\d}(\wx)$ is Poisson. Moreover, every compatible Poisson bracket on $\FFF_\C$ is a scalar multiple of $\Poi_V^\CC$ for some $V$.
\label{allcompat}
\end{proposition}

\section{Poisson-Lie groups and the main conjecture}
\label{SecMC}

\subsection{}
Let $\G$ be a Lie group equipped with a Poisson bracket $\Poi$.
$\G$ is called a {\em Poisson-Lie group\/}
if the multiplication map
$$
\G\times \G \ni (x,y) \mapsto x y \in \G
$$
is Poisson. Perhaps, the most important class of Poisson-Lie groups
is the one associated with classical R-matrices. 

Let $\g$ be the Lie algebra of $\G$
equipped with a nondegenerate invariant bilinear form $(\ ,\ )$, 
$\mathfrak{t}\in \g\otimes\g$ be the corresponding Casimir element.
For an arbitrary element $r=\sum_i a_i\otimes b_i\in\g\times\g$ denote
\[
[[r,r]]=\sum_{i,j} [a_i,a_j]\otimes b_i\otimes b_j+\sum_{i,j} a_i\otimes [b_i,a_j]\otimes b_j+
\sum_{i,j} a_i\otimes a_j\otimes [ b_i,b_j]
\]
and $r^{21}=\sum_i b_i\otimes a_i$.

A {\em classical R-matrix} is an element $r\in \g\otimes\g$ that satisfies
{\em the classical Yang-Baxter equation (CYBE)} 
\begin{equation}
[[r, r]] =0
\label{CYBE}
\end{equation}
together with the condition 
\begin{equation}
r + r^{21} = \mathfrak{t}.
\label{rskew}
\end{equation}

Given a solution $r$ to \eqref{CYBE}, one can construct explicitly the Poisson-Lie bracket on the Lie group $\G$.
Choose a basis $\{I_i\}$ in $\g$, and let $\partial^R_i$ and $\partial^L_i$ be the right
and the left invariant vector fields on $\G$ whose values at the unit element equal $I_i$. Write $r$ as
$r=\sum_{i,j} r_{ij}I_i\otimes I_j$, then the Poisson-Lie bracket on $\G$ is given by
\begin{equation}\label{sklya}
\{f_1,f_2\}=\sum_{i,j}r_{ij}\left(\partial^R_i f_1\partial^R_j f_2-
\partial^L_i f_1\partial^L_j f_2\right),
\end{equation}
see \cite[Proposition 4.1.4]{KoSo}. This bracket is called the {\it Sklyanin bracket\/} corresponding to $r$.

The classification of classical R-matrices for simple complex Lie groups was given by Belavin and Drinfeld in \cite{BD}. Let $\G$ be a simple complex Lie group,
$\g$ be the corresponding Lie algebra, $\h$ be its Cartan subalgebra,
$\Phi$ be the root system associated with $\g$, $\Phi^+$ be the set of positive roots, and $\Delta\subset \Phi^+$ be the 
set of positive simple roots. 
A {\em Belavin-Drinfeld triple} $T=(\Gamma_1,\Gamma_2, \gamma)$
consists of two subsets $\Gamma_1,\Gamma_2$ of $\Delta$ and an isometry $\gamma:\Gamma_1\to\Gamma_2$ nilpotent in the 
following sense: for every $\alpha \in \Gamma_1$ there exists $m\in\mathbb{N}$ such that $\gamma^j(\alpha)\in \Gamma_1$ for $j=0,\ldots,m-1$, but $\gamma^m(\alpha)\notin \Gamma_1$. The isometry $\gamma$ extends in a natural way to a map between root systems generated by $\Gamma_1, \Gamma_2$. This allows one to define a partial ordering on $\Phi$: $\alpha \prec_T \beta$ if $\beta=\gamma^j(\alpha)$ for some $j\in \mathbb{N}$. 

Select root vectors $e_\alpha \in\g$ satisfying 
$(e_{-\alpha},e_\alpha)=1$. According to the Belavin-Drinfeld classification, the following is true (see, e.g., \cite[Chap.~3]{CP}).

\begin{proposition}\label{bdclass}
{\rm(i)} Every classical R-matrix is equivalent {\rm(}up to an action of $\sigma\otimes\sigma$, where $\sigma$ is an 
automorphism of $\g$\/{\rm)} to the one of the form
\begin{equation}
\label{rBD}
r= r_0 + \sum_{\alpha\in \Phi^+} e_{-\alpha}\otimes e_\alpha + \sum_{\stackrel{\alpha \prec_T \beta}{\alpha,\beta\in\Phi^+}} e_{-\alpha}\wedge e_\beta.
\end{equation}

{\rm(ii)} $r_0\in \h\otimes\h$ in \eqref{rBD} satisfies 
\begin{equation}
(\gamma(\alpha)\otimes \Id )r_0 + (\Id\otimes \alpha )r_0 = 0  
\label{r01}
\end{equation}
for any $\alpha\in\Gamma_1$ and
\begin{equation}
r_0 + r_0^{21} = \mathfrak{t}_0,
\label{r02}
\end{equation}
where $\mathfrak{t}_0$ is the $\h\otimes\h$-component of $\mathfrak{t}$. 

{\rm(iii)} Solutions $r_0$ to \eqref{r01}, \eqref{r02}
form a linear space of dimension $\frac{k_T(k_T-1)}{2}$, where 
$k_T= | \Delta \setminus \Gamma_1 |$; more precisely, define
\begin{equation}
\h_T=\{ h\in\h \ : \ \alpha(h)=\beta(h)\ \mbox{if}\ \alpha\prec_T\beta\}, 
\label{htau}
\end{equation}
then $\dim\h_T=k_T$, and if $r_0'$ is a fixed solution of \eqref{r01}, \eqref{r02}, then
every other solution has a form $r_0=r_0' + s$, where $s$ is an arbitrary element of $\h_T\wedge\h_T$.
\end{proposition}

We say that two classical R-matrices that have a form \eqref{rBD} belong to the same {\em Belavin-Drinfeld class\/}
if they are associated with the same Belavin-Drinfeld triple.

\subsection{}
Let $\G$ 
be a simple complex Lie group. Given a Belavin-Drinfeld triple $T$ for $\G$,
define the torus $\H_T=\exp \h_T\subset\G$.

We conjecture that there exists a classification of regular cluster structures on $\G$ that is completely
parallel to the Belavin-Drinfeld classification.

\begin{conjecture}
\label{ulti}
Let $\G$ be a simple complex Lie group.
For any Belavin-Drinfeld triple $T=(\Gamma_1,\Gamma_2,\gamma)$ there exists a cluster structure
$(\CC_T,\varphi_T)$ on $\G$ such that

{\rm (i)} 
the number of stable variables is $2k_T$, and the corresponding extended exchange matrix has a full rank; 

{\rm (ii)} $(\CC_T,\varphi_T)$ is regular, and the corresponding upper cluster algebra $\UU_\C(\CC_T)$ 
is naturally isomorphic to $\O(\G)$;

{\rm (iii)} the global toric action of $(\mathbb{C}^*)^{2k_T}$ on $\C(\G)$ is generated by the action
of $\H_T\times \H_T$ on $\G$ given by $(H_1, H_2)(X) = H_1 X H_2$;

 {\rm (iv)} for any solution of CYBE that belongs to the Belavin-Drinfeld class specified  by $T$, the corresponding Sklyanin bracket is compatible with $\CC_T$;

{\rm (v)} a Poisson-Lie bracket on $\G$ is compatible with $\CC_T$ only if it is a scalar multiple
of the Sklyanin bracket associated with a solution of CYBE that belongs to the Belavin-Drinfeld class specified  by $T$.
\end{conjecture}

\begin{remark}\label{tam}
Let us explain the meaning of assertion (iii) of Conjecture 
\ref{ulti} in more detail. For any $H\in\H$ and any weight $\omega\in\h^*$ put  $H^\omega=e^{\omega(h)}$, whenever
$H=\exp h$. 
Let $(\wx, \wB)$ be a seed in $\CC_T$, and $y_i=\varphi(x_i)$ for $i\in [n+m]$.
Then (iii) is equivalent to the following:

1) for any $H_1, H_2 \in \H_T$ and any $X\in\G$,
$$
y_i(H_1 X H_2)= H_1^{\eta_i} H_2^{\zeta_i} y_i(X)
$$
for some weights $\eta_i,\zeta_i \in \h_T^*$ ($i\in[n+m]$);

2)
$\mbox{span}\{\eta_i\}_{i=1}^{\dim\G} = \mbox{span}\{\zeta_i\}_{i=1}^{\dim\G}=\h^*_T$;

3) for every $i\in [\dim\G-2k_T]$,
$$
\sum_{j=1}^{\dim\G} b_{ij}\eta_j = \sum_{j=1}^{\dim\G} b_{ij}\zeta_j = 0.
$$
\end{remark}

\section{Towards a proof of the main conjecture}\label{reduction}
The goal of this sections is to prove

\begin{theorem}
\label{partial}
Let $T=(\Gamma_1, \Gamma_2,\gamma)$ be a Belavin-Drinfeld triple and $(\CC_T,\varphi_T)$ be a cluster structure
on $\G$.
Suppose that assertions {\rm(i)} and {\rm(iii)} of Conjecture {\rm\ref{ulti}} are valid and that 
assertion {\rm(iv)} is valid for one particular R-matrix in the Belavin-Drinfeld 
class specified  by $T$. Then {\rm(iv)} and {\rm(v)} are valid for the whole Belavin-Drinfeld class specified  
by $T$.
\end{theorem}

\begin{proof}
We start with the following auxiliary statement.

\begin{lemma}
\label{Adr}
Any R-matrix from the Belavin-Drinfeld class specified  by $T$ is invariant under the adjoint action
of $\H_T\otimes\H_T$.
\end{lemma}
\begin{proof} Fix an arbitrary $H\in\H_T$.
The term $r_0$ in (\ref{rBD}) is clearly fixed by $\Ad_{H}\otimes \Ad_{H}$.
Furthermore, 
\[
\Ad_{H}\otimes \Ad_{H} (e_{-\alpha} \otimes e_{\alpha}) = H^{-\alpha} e_{-\alpha}\otimes H^\alpha e_{\alpha}=e_{-\alpha} \otimes e_{\alpha}.
\]
Besides, for $\alpha \prec_T \beta$, 
\[ 
\Ad_{H}\otimes \Ad_{H} (e_{-\alpha}\wedge e_\beta) = H^{\beta-\alpha} e_{-\alpha}\wedge e_\beta= 
e_{-\alpha}\wedge e_\beta, 
\]
since $\beta-\alpha$ annihilates $\h_T$.
\end{proof}

Our plan is to invoke the construction used in  Proposition~\ref{allcompat}, so the first
goal is to define a Poisson structure on the torus $H_T\times \H_T$ satisfying~\eqref{torbra}.
Let $V_1, V_2: \h^*_T \to \h_T$ be two linear skew-symmetric maps,
that is  $\langle \eta, V_i \zeta \rangle = - \langle \zeta, V_i \eta \rangle$ for any
$\eta,\zeta\in\h_T^*$ and $i=1,2$, where $\langle\cdot,\cdot\rangle$ stands for the natural coupling 
between $\h^*_T$ and $\h_T$. Besides, let $V_{12}$ be an  
arbitrary linear map $\h^*_T \to \h_T$. Put
$$
V=\left (
\begin{array}{cc}
V_1 & V_{12}\\
-V_{12}^* & V_2
\end{array}
\right )
: \h^*_T \oplus \h^*_T \to \h_T\oplus \h_T.
$$
Then one can define a Poisson structure $\Poi_V$ on $\H_T\times\H_T$  
by the formula
\begin{equation}
\label{HHbrack}
\{\varphi_1,\varphi_2\}_V = \langle V D \varphi_1, D \varphi_2 \rangle,
\end{equation}
where the differential $D \varphi$ is given by
$$
\langle D \varphi (H_1, H_2), \eta \oplus \zeta\rangle = \left.\frac{d}{dt}\right\vert_{t=0} \left ( \varphi(e^{t\eta}H_1,           
H_2e^{t\zeta})\right ).
$$
In particular, the Poisson bracket of ``monomial" functions on $\H_T\times\H_T$ is given by
\begin{multline*}
\left \{ H_1^{\eta_1} H_2^{\eta_2} , H_1^{\zeta_1} H_2^{\zeta_2} 
\right  \}_V =\\
\left (
\langle V_1\eta_1, \zeta_1 \rangle + \langle V_2\eta_2, \zeta_2  
\rangle
+
\langle V_{12} \eta_1, \zeta_2 \rangle + \langle V_{12}\eta_2,  
\zeta_1 \rangle
\right )
H_1^{\eta_1} H_2^{\eta_2} H_1^{\zeta_1} H_2^{\zeta_2}.
\end{multline*}
By choosing appropriate $\eta_1$, $\eta_2$, $\zeta_1$, $\zeta_2$ in the above relation we make certain that
$\Poi_V$ satisfies~\eqref{torbra}.

Fix an R-matrix $r$ in the Belavin-Drinfeld class specified by $T$  
and denote by
$\Poi_r$ the corresponding Sklyanin bracket.
It will be convenient to rewrite formula (\ref{sklya}) for $\Poi_r$ as
\[
\{f_1,f_2\}_r = \langle R(d^R f_1), d^R f_2 \rangle - \langle R(d^L f_1), d^L f_2 \rangle,
\]
where $R: \h^* \to \h$ is given by $\langle R\eta, \zeta \rangle = \langle r, \eta\otimes\zeta \rangle$.
We will view $\M=\H_T\times\H_T\times\G$ as a direct product of Poisson  
manifolds
$\left ( \H_T\times\H_T,\Poi_V \right )$ and $\left(\G,\Poi_r \right )$.
Consider the map $\mu:\M\ni (H_1, H_2, X) \mapsto H_1X  
H_2\in\G$.

\begin{lemma}
\label{rSbrack}
{\rm(i)} The map $\mu$ induces a Poisson bracket $\Poi_{r,V}$ on $\G$ given by 
$$
\{f_1, f_2\}_{r,V}=\{f_1, f_2\}_r+\left \langle V \left((d^R f_1 )_0\oplus (d^L f_1)_0\right), (d^R f_2)_0\oplus 
(d^L f_2)_0\right \rangle,
$$
where $(\cdot)_0$ stands for the projection onto $\h^*$.
 
{\rm(ii)} The bracket $\Poi_{r,V}$ is Poisson-Lie if and only if $V_{12}=0$ and $V_2=-V_1$.
\end{lemma}

\begin{proof}
Let $f$ be a function on $\G$. For any fixed $(H_1,H_2)\in \H_T\times \H_T$ define the function $f^{H_1,H_2}$ on
$\G$ via $f^{H_1,H_2}(X)=f\circ\mu(H_1,H_2,X)$. Similarly, for any fixed $X\in \G$ define the function $f^X$ on
$\H_T\times \H_T$ via $f^X(H_1,H_2)= f\circ\mu(H_1,H_2,X)$.
Given two functions $f_1,f_2$ on $\G$, let us compute  the  following Poisson bracket on $\M$:
\begin{equation}\label{rS}
\{f_1\circ\mu, f_2\circ\mu\}(H_1,H_2,X) =\\
 \{f_1^{H_1, H_2}, f_2^{H_1, H_2}\}_r (X) +
\{f_1^X, f_2^X\}_V (H_1, H_2).
\end{equation}

First observe that for a function $f$ on $\G$ 
$$
d^L f^{H_1, H_2}(X) = \Ad^*_{H_1} d^L f(H_1X H_2), \qquad d^R f^{H_1, H_2}(X) = \Ad^*_{H_2^{-1}} d^R f(H_1X H_2).
$$
Since, by Lemma \ref{Adr}, $\Ad_H \circ R\circ \Ad^*_H = R$ for any $H\in \H_T$, 
this means that the first term in the right-hand side of \eqref{rS} is
equal to 
$\{f_1, f_2\}_r(H_1X H_2)$.

To compute the second term in \eqref{rS}, note that
$$
D f^X( H_1, H_2) =  \left(d^R f(H_1X H_2)\right )_0\oplus \left(d^L f(H_1X H_2)\right )_0.
$$
Then it follows from
\eqref{HHbrack} and \eqref{rS} that
$$
\{f_1\circ\mu, f_2\circ\mu\} =\left (\{f_1, f_2\}_r+\left \langle V \left((d^R f_1)_0\oplus (d^L f_1)_0\right), 
(d^R f_2)_0\oplus (d^L f_2)_0\right \rangle\right )\circ\mu,
$$
which proves the first claim of the lemma. 

Conditions on $V$ that ensure that $\Poi_{r,V}$ is Poisson-Lie
drop out immediately from the fact that any Poisson-Lie structure is trivial at the identity of $\G$.
\end{proof}

We can now proceed with the proof of the theorem.
Assertion (i) guarantees that the toric action mentioned in assertion (iii) has the maximal rank.
Assume that assertion (iii) of Conjecture~\ref{ulti} is valid. 
Then claims 1) and 2) of Remark~\ref{tam}, together with Lemma \ref{rSbrack}(i) and 
Proposition \ref{allcompat} imply that
if $\Poi_r$ is compatible with the cluster structure $\CC_T$, then every other compatible structure
is of the form $\Poi_{r,V}$ for some choice of $V$. 
Since $\varphi_T(\wx)$ defines a coordinate chart on $\G$, we
conclude that any Poisson bracket on $\G$ compatible with $\CC_T$ is, in fact, a scalar multiple
of $\Poi_{r,V}$.
Moreover, by Lemma \ref{rSbrack}(ii), $\Poi_{r,V}$ is Poisson-Lie
 if and only if it can be written in the form (\ref{sklya}) with $r$ replaced by $r+s$, where $s$ is 
an arbitrary element of $\h_T\wedge\h_T$. But this is exactly the description of the Belavin-Drinfeld class
specified by $T$. The proof is complete. 
\end{proof}

\section{Evidence supporting the conjecture}

Here we discuss several instances in which Conjecture \ref{ulti} has been verified.

\subsection{The trivial Belavin-Drinfeld data}\label{Secgen} 
The Belavin-Drinfeld data (triple, class) is said to be {\it trivial\/} if $\Gamma_1=\Gamma_2=\varnothing$.
In this case we use subscript $0$ instead of $T$, so $k_0=|\Delta|$ is the rank of $\G$ and $\H_0=\H$ is the Cartan
subgroup in $\G$. 

\begin{theorem}
\label{triv}
Let $\G$ be a simple complex Lie group of rank $n$, then
there exists a cluster structure
$(\CC_0,\varphi_0)$ on $\G$ such that

{\rm (i)} 
the number of stable variables is $2n$, and the corresponding extended exchange matrix has full rank; 

{\rm (ii)} $(\CC_0,\varphi_0)$ is regular, and the corresponding upper cluster algebra $\UU_\C(\CC_0)$  is naturally isomorphic to $\O(\G)$;

{\rm (iii)} the global toric action of $(\mathbb{C}^*)^{2n}$ on $\C(\G)$ is generated by the action
of $\H\times \H$ on $\G$ given by $(H_1, H_2)(X) = H_1 X H_2$;

 {\rm (iv)} for any solution of CYBE that belongs to the trivial Belavin-Drinfeld class, the corresponding Sklyanin bracket is compatible with $\CC_0$;

{\rm (v)} a Poisson-Lie bracket on $\G$ is compatible with $\CC_0$ only if it is a scalar multiple
of the Sklyanin bracket associated with a solution of CYBE that belongs to the trivial Belavin-Drinfeld class.
\end{theorem}

\begin{proof}
By Theorem~\ref{partial}, we have to prove assertions (i)--(iii) and exhibit one bracket satisfying (iv). 
As we have mentioned in the Introduction, 
paper \cite{CAIII} suggests a construction of a cluster structure  
on the double Bruhat cell $\G^{u,v}$ for an arbitrary pair of elements $u,v$ in the Weyl group $W$ of $\G$.
Let $u=v=w_0$ be the longest element of $W$, then the corresponding double
Bruhat cell 
is open and dense in $\G$, and hence the construction in \cite{CAIII}
gives rise to a cluster structure on $\G$. We claim that this cluster structure, which we denote by $(\CC_0,\varphi_0)$, 
satisfies all conditions of the theorem.

We start with a brief review  of the construction in \cite{CAIII}. First, following \cite{FZBruhat}, 
let us recall the definition of {\em generalized minors\/} in $\G$.
Let $\N_+$ and $\N_-$ be the upper and the lower maximal unipotent subgroups of $\G$. 
For every $X$ in an open Zarisky dense subset
\[
\G^0=\N_- \H \N_+
\]
of $\G$ there exists a unique {\em Gauss factorization}
$$
X = X_- X_0 X_+, \quad X_+ \in \N_+, \ X_- \in \N_-, \  X_0 \in \H.
$$
For any $X\in \G^0$ and a fundamental weight $\omega_i\in \h^*$ define
\[
\Delta_i(X)=X_0^{\omega_i};
\]
this function can be extended to a regular function on the whole group $\G$. 
For any pair $u,v\in W$, the corresponding
{\em generalized minor\/} is a regular function on $\G$ given by
\begin{equation}\label{genminor}
\Delta_{u\omega_i,v\omega_i}(X)=\Delta_i( u^{-1}X v).
\end{equation}
These functions depend only on the weights $u \omega_i$ and $v\omega_i$, and
do not depend on the particular choice of $u$ and $v$.

 The initial cluster for $(\CC_0,\varphi_0)$ can be chosen as a certain collection of generalized minors, described as follows. 
 Consider two reduced words for $w_0$, one written in
the alphabet $1, \ldots, n$ and another, in the alphabet $-1, \ldots, -n$. Let $\ii= (i_l)$ be a word of length
$2 l(w_0) + n = \dim \G$ defined as a shuffle of the two reduced words above appended with a string
$(-n, \ldots, -1)$ on the left.
For $k\in [2 l(w_0)]$ denote
\[
u_{\le k}=u_{\le k}(\ii)=\prod_{l=1,\dots,k} s_{|i_l|}^{\frac{1-\sign(i_l)}2},\qquad
v_{>k}=v_{>k}(\ii)=\prod_{l=2 l(w_0),\dots, k+1} s_{|i_l|}^{\frac{1+\sign(i_l)}2}.
\]
Besides, for $k\in -[n]$ set $u_{\le k}$ to be the identity and $v_{>k}$ to be equal to
$w_0$. For $k\in -[n] \cup [2 l(w_0)]$ put
\[
\Delta(k;\ii)=\Delta_{u_{\le 
k}\omega_{|i_{
k}|},v_{>
k}\omega_{|i_{
k}|}},
\]
where the right hand side is the generalized minor defined by~\eqref{genminor}. 
Then $\tilde \x= \left (x_{k,\ii}=\Delta(k;\ii)\ : \ k\in -[n] \cup [2 l(w_0)] \right )$ is an extended cluster in
$(\CC_0,\varphi_0)$ with
$2 n = 2 \dim \H$ stable variables given by $ \Delta(k;\ii)$,  $k\in - [ n]$, and  
$ \Delta(k_j;\ii)$,  $j\in [n]$, where $k_j \in [2 l(w_0)]$ is the largest index such that $|i_{k_j}|=j$.
The matrix $\tilde B=(b_{ij})$ for the seed associated with $\tilde \x$ can be described explicitly in terms of the word $\ii$, however we will not need this description here. By Proposition 2.6 in \cite{CAIII},
it has full rank.

So, assertion (i) of the theorem is proved.

To prove assertion (ii), observe that $\O(\G^{w_0,w_0})$ is obtained from $\O(\G)$ via localization at stable 
variables $\Delta(k;\ii)$  and  $\Delta(k_j;\ii)$ defined above. Besides, by Theorem 2.10 in \cite{CAIII}, 
$\O(\G^{w_0,w_0})$ is naturally isomorphic to the upper cluster 
algebra $\UU(\CC_0)$ over $\C\P$, where $\P$ is generated by Laurent monomials in these stable variables.
The latter is obtained via localization at the same stable 
variables from the upper cluster algebra $\UU_\C(\CC_0)$, and hence (ii) follows.

To establish (iii), we need to check claims 1)--3) of Remark~\ref{tam}. Let $H_1, H_2$ be two elements in $\H$.
We want to compute the local toric action 
$\TE^W_{H_1,H_2}$
on $\tilde \x$ generated by the action of $\H\times\H$ on $\G$.  Clearly, 
$\Delta_i(H_1 X H_2)=(H_1 X_0 H_2)^{\omega_i}$, hence
\begin{multline*}
\Delta_{u\omega_i,v\omega_i}(H_1 X H_2)=\Delta_i\left ( (u^{-1}H_1 u) (u^{-1} X v) (v^{-1}H_2 v)\right )\\
= (u^{-1}H_1 u)^{\omega_i} (v^{-1}H_2 v)^{\omega_i} \Delta_{u\omega_i,v\omega_i}(X)
= H_1^{u\omega_i} H_2^{v\omega_i} \Delta_{u\omega_i,v\omega_i}(X).
\end{multline*}
Thus,
$$
\TE^W_{H_1,H_2}(\wx)=\left ( x_{k,\ii} H_1^{u_{\le k}\omega_{|i_k|}} H_2^{v_{>k}\omega_{|i_k|}}\right )_{k \in -[n] \cup [1, 2 l(w_0)]},
$$
where rows of $W$ are given by components of weights $u_{\le k}\omega_{|i_k|} ,\ v_{>k}\omega_{|i_k|}$ with respect to
some basis in $\h^*$, which amounts to claim 1).  
Claim 2) follows from the fact that, for $k \in -[n]$,
exponents of $H_1$ and $H_2$ in the formula above
are $\omega_1,\ldots,\omega_n$ and $w_0\omega_1,\ldots,w_0\omega_n$, respectively, and
each of these two collections spans $\h^*$. Finally
claim 3) that guarantees that 
$\TE^W_{H_1,H_2}$ extends to a global toric action amounts to equations
$$
\sum_{k \in -[n] \cup [1, 2 l(w_0)]} b_{lk} u_{\le k}\omega_{|i_k|}= \sum_{k \in -[n] \cup [1, 2 l(w_0)]} b_{lk} v_{> k}\omega_{|i_k|} = 0 
$$
for $l \in [1, 2 l(w_0)], l \ne k_j\ (j=1,\ldots,n)$. But this is precisely the statement of Lemma 4.22 in \cite{GSVb}, which  proves (iii).

It remains to exhibit a Poisson structure on $\G$ corresponding to the trivial Belavin-Drinfeld data and compatible
with $\CC_0$. An immediate modification of Theorem~4.18 in \cite{GSVb} shows that the standard Poisson-Lie structure on $\G$ satisfies these requirements.
\end{proof}

\subsection{The case of $SL_n$ for $n=2,3,4$}\label{Sec34}
In this Section we prove the following result.

\begin{theorem} Conjecture~{\rm\ref{ulti}} holds for complex Lie groups $SL_2$, $SL_3$ and  $SL_4$.
\end{theorem} 

\begin{proof}
The case of $SL_2$ is completely covered by Theorem~\ref{triv}, since in this case $\Delta$ contains only one 
element, and hence the only Belavin-Drinfeld triple is the trivial one.

Before we move on to the case of $SL_3$, consider the following two isomorphisms of the Belavin-Drinfeld data for $SL_n$ 
(here $n$ is arbitrary): the first one transposes $\Gamma_1$ and $\Gamma_2$ and reverses the direction of $\gamma$, while the 
second one takes each root $\alpha_j$ to $\alpha_{w_0(j)}$. Clearly, these isomorphisms correspond to the automorphisms 
of $SL_n$ given by $X\mapsto -X^t$ and $X\mapsto w_0Xw_0$.   Since we consider R-matrices up to an action of
 $\sigma\otimes\sigma$, in what follows we do not distinguish between Belavin-Drinfeld triples obtained one from the other 
via the above isomorphisms.

In the case of $SL_3$ we have $\Delta=\{\alpha_1,\alpha_2\}$, and hence, up to an isomorphism, there 
is only one non-trivial Belavin-Drinfeld 
triple: $T=(\Gamma_1=\{\alpha_2\}, \Gamma_2=\{\alpha_1\}, \gamma:\alpha_2\mapsto\alpha_1)$. 
In this case $k_T=1$, and hence, by Proposition~\ref{bdclass}(iii), the
corresponding Belavin-Drinfeld class contains a unique R-matrix. It is called the
{\em Cremmer-Gervais R-matrix\/}, and the solution to~\eqref{r01},~\eqref{r02} is given by
$$
r_0 -\frac12 \mathfrak{t}_0=  \frac 1 6 \left (e_{11}\wedge e_{33} - e_{11}\wedge e_{22} - e_{22}\wedge e_{33}\right )
$$
(see e.g.~\cite{GeGi}).

To prove Conjecture~\ref{ulti} in this case, we once again rely on Theorem~\ref{partial}. 
Let us define the cluster structure $(\CC_{\text{CG}},\varphi_{\text{CG}})$ validating assertion (i) of the conjecture. 
Since $\dim\G=8$, the extended exchange matrix $\wB_{\text{CG}}$ should have 6 rows and 8 columns. Put
\[
\wB_{\text{CG}}=\left(\begin{array}{rrrrrrrr}
0 & -1 & -1 & 1 & 0 & 0 & 0 & 0 \\
1 & 0 & -1 & -1 & 0 & 0 & 1 & 0 \\
1 & 1 & 0 & 0 & 1 & -1 & -1 & 0 \\
-1 & 1 & 0 & 0 & 1 & 1 & 0 & -1 \\
0 & 0 & -1 & -1 & 0 & 1 & 0 & 1 \\
0 & 0 & 1 & -1 & -1 & 0 & 0 & 0 
\end{array}\right).
\]
It is easy to check that $\rank \wB_{\text{CG}}=6$. So, to establish (i) it remains to define the field 
isomorphism $\varphi_{\text{CG}}$.

Let $X=(x_{ij})_{i,j=1}^3$ be a matrix in $SL_3$, $\widehat{X}=(\hat x_{ij})_{i,j=1}^3$ be its adjugate
matrix given by $\widehat{X}=X^{-1}\det X$. Given the initial extended cluster $(x_1,\dots,x_8)$, denote
$P_i=\varphi_{\text{CG}}(x_i)$ and put
\begin{align}\label{formaple}
P_1=x_{11},\qquad P_2=x_{13}&, \qquad P_3=x_{21}, \notag\\
P_4=-\hat{x}_{23}, \qquad P_5=-\hat{x}_{31}&, \qquad P_6=-\hat{x}_{33},\\
P_7=x_{13}x_{31}-x_{21}x_{23}, \qquad &P_8=\hat x_{13}\hat x_{31}-\hat x_{21}\hat x_{23}. \notag
\end{align}
A direct computation shows that gradients of $P_i$, $1\le i\le 8$, are linearly independent at a generic point of 
$SL_3$, hence $(P_1,\dots,P_8)$ form a transcendence basis of $\C(SL_3)$, 
and assertion (i) is established.

The proof of assertion (ii) relies on Proposition~\ref{regfun}. Since $SL_n$ (and, in particular, $SL_3$) is not
a Zariski open subset of $\C^k$, we extend the cluster structure $(\CC_{\text{CG}},\varphi_{\text{CG}})$ to a
cluster structure $(\CC_{\text{CG}}',\varphi_{\text{CG}}')$ on $GL_3$ by appending the column $(0, 0, 0, 0, -1, 1)^t$
on the right of the matrix $\wB_{\text{CG}}$ and adding the function $P_9=\det X$ to the initial cluster. 
Conditions (i), (ii) and (iv) of Proposition~\ref{regfun} for $(\CC_{\text{CG}}',\varphi_{\text{CG}}')$ are clearly 
true, and condition (iii) is verified by direct computation. The ring of regular functions on $GL_3$ is generated by
the matrix entries $x_{ij}$. By Theorem~3.21 in \cite{GSVb}, condition (i) implies
that the upper cluster algebra coincides with the intersection of rings of Laurent polynomials in cluster variables
taken over the initial cluster and all its adjacent clusters. Therefore, to check condition (v) of 
Proposition~\ref{regfun}, it suffices to check that every matrix entry can be written as a Laurent polynomial in
each of the seven clusters mentioned above. This fact is verified by direct computation with Maple: we solve 
system of equations~\eqref{formaple}, as well as six similar systems, with respect to $x_i$. Since $GL_3$ is Zariski
open in $\C^9$, $\UU_\C(\CC_{\text{CG}}')$ is naturally isomorphic to $\O(GL_3)$ by Proposition~\ref{regfun}.
Now assertion (ii) for $(\CC_{\text{CG}},\varphi_{\text{CG}})$ follows from the fact that both 
$\UU_\C(\CC_{\text{CG}})$ and $\O(SL_3)$ are obtained from their $GL_3$-counterparts via restriction to $\det X=1$. 

To prove assertion (iii), we parametrize the left and the right action of $\H_{\text{CG}}$ by $\diag(t,1,t^{-1})$
and $\diag(z,1,z^{-1})$, respectively. Then
\begin{multline*}
\left(\begin{array}{ccc}
t & 0 & 0\\
0 & 1 & 0\\
0 & 0 & t^{-1}
\end{array}\right)
\left(\begin{array}{ccc}
x_{11} &  x_{12}& x_{13}\\
 x_{21}&  x_{22}& x_{23}\\
 x_{31}& x_{32} & x_{33}
\end{array}\right)
\left(\begin{array}{ccc}
z & 0 & 0\\
0 & 1 & 0\\
0 & 0 & z^{-1}
\end{array}\right)=\\
\left(\begin{array}{ccc}
tzx_{11} &  tx_{12}& tz^{-1}x_{13}\\
 zx_{21}&  x_{22}& z^{-1}x_{23}\\
t^{-1}z x_{31}& t^{-1}x_{32} & t^{-1}z^{-1}x_{33}
\end{array}\right),
\end{multline*}
and hence condition 1) of Remark~\ref{tam} holds with $1$-dimensional vectors $\eta_i$, $\zeta_i$ given by
\begin{align*}
\eta_1=1,\quad \eta_2=1, \quad \eta_3=0, \quad &\eta_4=1,\quad \eta_5=-1, \quad \eta_6=1, \quad 
\eta_7=0,\quad \eta_8=0,\\
\zeta_1=1,\quad \zeta_2=-1, \quad \zeta_3=1, \quad &\zeta_4=0,\quad \zeta_5=1, \quad \zeta_6=1, \quad 
\zeta_7=0,\quad \zeta_8=0.
\end{align*}
Conditions 2) and 3) are now verified via direct computation.

Finally, let us check that assertion (iv) holds for the Cremmer-Gervais bracket. A direct computation shows that
this bracket in the basis $(P_1,\dots,P_8)$ satisfies \eqref{cpt}, and that the corresponding coefficient
matrix is given by
\[
3\Omega=
\left(\begin{array}{rrrrrrrr}
0 & -2 & -2 & -1 & -1 & 0 & -3 & -3 \\
2 & 0 & 0 & 0 & 0 & 1 & -2 & -1 \\
2 & 0 & 0 & 0 & 0 & 1 & 1 & -1 \\
1 & 0 & 0 & 0 & 0 & 2 & -1 & -2 \\
1 & 0 & 0 & 0 & 0 & 2 & -1 & 1 \\
0 & -1 & -1 & -2 & -2 & 0 & -3 & -3 \\
3 & 2 & -1 & 1 & 1 & 3 & 0 & 0 \\
3 & 1 & 1 & 2 & -1 & 3 & 0 & 0 
\end{array}\right).
\]
A direct check shows that $\wB_{\text{CG}}\Omega=(-I\; 0)$, hence the
Cremmer-Gervais bracket is compatible with $\CC_{\text{CG}}$ by Proposition~\ref{Bomega}.

\begin{remark}\label{reallife}
Although we started the presentation above by specifying $\wB$ and $\wx$,
to construct the cluster structure  $(\CC_{\text{CG}},\varphi_{\text{CG}})$ we had to act the other way around. 
We started with the extension of the Cremmer-Gervais bracket to $GL_3$ and tried to find a regular basis in
$\C(GL_3)$ in which this bracket is diagonal quadratic (that is, satisfies~\eqref{cpt}). 
Since $\det X$ is a Casimir  function for the extended bracket, it was included in the basis from the very beginning
as a stable variable.
Once such a basis is built, the exchange matrix of the cluster structure $\CC_{\text{CG}}^\circ$
on $GL_3$ is restored via 
Proposition~\ref{Bomega}. The cluster structure on $SL_3$ is obtained via restriction to the hypersurface $\det X=1$,
which amounts to removal of the corresponding column of the exchange matrix. 
\end{remark}

Let us proceed with the case of $SL_4$. Here we have, up to isomorphisms, the following possibilities:

Case 1: $\Gamma_1=\Gamma_2=\varnothing$ (standard R-matrices); 

Case 2: $\Gamma_1=\{\alpha_2, \alpha_3\}$, $\Gamma_2=\{\alpha_1, \alpha_2\}$, $\gamma(\alpha_2)=\alpha_1$,
$\gamma(\alpha_3)=\alpha_2$ (Cremmer-Gervais R-matrix);

Case 3: $\Gamma_1=\{\alpha_1\}$, $\Gamma_2=\{\alpha_3\}$, $\gamma(\alpha_1)=\alpha_3$;

Case 4: $\Gamma_1=\{\alpha_1\}$, $\Gamma_2=\{\alpha_2\}$, $\gamma(\alpha_1)=\alpha_2$.

The first case is covered by Theorem~\ref{triv}. In the remaining cases we proceed in accordance with 
Remark~\ref{reallife}.

{\it Case 2}. Here $k_T=1$, and hence the corresponding Belavin-Drinfeld class contains a unique R-matrix. It is called the
{\em Cremmer-Gervais R-matrix\/}, and the solution to~\eqref{r01},~\eqref{r02} is given by
$$
r_0 -\frac12 \mathfrak{t}_0= \frac14\left(e_{11}\wedge e_{44}-e_{11}\wedge e_{22}-e_{22}\wedge e_{33}-e_{33}\wedge e_{44}\right).
$$
The basis in $\C(SL_4)$ that makes the Cremmer-Gervais bracket diagonal quadratic is given by
\begin{align*}
&P_1= -x_{21}, \qquad P_2=x_{31},\qquad P_3=x_{24},\\ 
&P_4=\hat x_{31},\qquad P_5=\hat x_{24},\qquad P_6=\hat x_{34},\\
&P_7=\left|\begin{array}{cc} x_{11}& x_{14}\\ x_{21} & x_{24}\end{array}\right|,\qquad
P_8=\left|\begin{array}{cc} x_{21}& x_{24}\\ x_{31} & x_{34}\end{array}\right|,\\
&P_9=\left|\begin{array}{cc} x_{21}& x_{14}\\ x_{31} & x_{24}\end{array}\right|,\qquad
P_{10}=\left|\begin{array}{cc} x_{21}& x_{22}\\ x_{31} & x_{32}\end{array}\right|,\\
&P_{11}=-\left|\begin{array}{cc} \hat x_{31}& \hat x_{24}\\ \hat x_{41} & \hat x_{34}\end{array}\right|,\qquad
P_{12}=-\left|\begin{array}{ccc} x_{21}& x_{22} & x_{14}\\ x_{31} & x_{32} & x_{24}\\
x_{41} & x_{42} & x_{34} \end{array}\right|,\\
&P_{13}=\sum_{i=1}^3 \hat x_{i+1,1} \left|\begin{array}{cc} x_{1i}& x_{14}\\ x_{2i} & x_{24}\end{array}\right|,\qquad
P_{14}=-\sum_{i=1}^3 \hat x_{i4}\left|\begin{array}{ccc} x_{21}& x_{2, i+1} & x_{14}\\ x_{31} & x_{3,i+1} & x_{24}\\
x_{41} & x_{4,i+1} & x_{34} \end{array}\right|,\\
&P_{15}=-\sum_{i=1}^3 \hat x_{i+1,1}\left|\begin{array}{ccc} x_{21}& x_{1i} & x_{14}\\ x_{31} & x_{2i} & x_{24}\\
x_{41} & x_{3i} & x_{34} \end{array}\right|,
\end{align*}
where $X=(x_{ij})_{i,j=1}^4$ is a matrix in $SL_4$ and  $\widehat{X}=(\hat x_{ij})_{i,j=1}^4$ is its adjugate
matrix. The coefficient matrix of the Cremmer-Gervais bracket in this basis is given by
\begin{multline*}
4\Omega=\\
\left(\begin{array}{rrrrrrrrrrrrrrr}
 0 & -3 & -3 & -1 & -1 & 0 & 0 & -2 & -3  & 0  & -1 & -2 & -2 & -4 & -4 \\
 3 & 0 & 0 & 0 & 0 & 1 & 2 & 0 & -1  & 2  & 1 & -1 & 1 & -2 & -2 \\
3 & 0 & 0 & 0 & 0 &  1 & 2 & 0 & 3  & 2  & 1 & 3 & 1 & 2 & 2 \\
1 & 0 & 0 & 0 & 0 &  3 & 2 & 0 & 1  & 2  & 3 & 1 & 3 & 2 & 2 \\
1 & 0 & 0 & 0 & 0 & 3 &  2 & 0 & 1  & 2  & -1 & 1 & -1 & -2 & -2 \\
0 & -1 & -1 & -3 & -3 & 0 & 0 & -2 & -1  & 0  & -3 & -2 & -2 & -4 & -4 \\
0 & -2 & -2 & -2 & -2 & 0 & 0 & -4 & -2  & 0  & -2 & 0 & -4 & -4 & -4 \\
2 & 0 & 0 & 0 & 0 & 2 & 4 & 0 & 2  & 4  & 2 & 2 & 2 & 0 & 0 \\
3 & 1 & -3 & -1 & -1 & 1 & 2 & -2 & 0  & 2  & 0 & 1 & -1 & -2 & -2 \\
0 & -2 & -2 & -2 & -2 & 0 & 0 & -4 & -2  & 0  & -2 & -4 & 0 & -4 & -4 \\
1 & -1 & -1 & -3 & 1 & 3 & 2 & -2 & 0  & 2  & 0 & -1 & 1 & -2 & -2 \\
2 & 1 & -3 & -1 & -1 & 2 & 0 & -2 & -1  & 4  & 1 & 0 & 0 & 0 & 0 \\
2 & -1 & -1 & -3 & 1 & 2 & 4 & -2 & 1  & 0  & -1 & 0 & 0 & 0 & 0 \\
4 & 2 & -2 & -2 & 2 & 4 & 4 & 0 & 2  & 4  & 2 & 0 & 0 & 0 & 0 \\
4 & 2 & -2 & -2 & 2 & 4 & 4 & 0 & 2  & 4  & 2 & 0 & 0 & 0 & 0 
\end{array}\right).
\end{multline*}
A basis in $\C(GL_4)$ is obtained by adding $P_{16}=\det X$ to the above basis. Since $\det X$ is a Casimir function,
the corresponding coefficient matrix $\Omega^\circ$ is obtained from $\Omega$ by adding a zero column on the right 
and a zero row at the bottom.
By assertion (i) of Conjecture~\ref{ulti}, the cluster structure $\CC_{\text{CG}}$ we are looking for  
should have 2 stable variables; their  images under $\varphi_{\text{CG}}$ are 
polynomials $P_{14}$ and $P_{15}$; recall that $P_{16}$ is the image of the third stable variable that exists 
in $\CC_{\text{CG}}^\circ$, but not in $\CC_{\text{CG}}$. Therefore, the exchange matrix of $\CC_{\text{CG}}^\circ$ is a $13\times 16$ matrix. It is given by
\begin{multline*}
\wB_{\text{CG}}^\circ=\\
\left(
\begin{array}{rrrrrrrrrrrrrrrr}
 0 \!&\! 1 \!&\! 1 \!&\! 0 \!&\! 0 \!&\! 0 \!&\! 0 \!&\! -1 \!&\! 0  \!&\! 0  \!&\! 0 \!&\! 0 \!&\! 0 \!&\! 0 \!&\! 0 \!&\! 0 \\
-1 \!&\! 0 \!&\! -1 \!&\! 0 \!&\! 0 \!&\! 0 \!&\! 0 \!&\! 0 \!&\! 1  \!&\! 0  \!&\! 0 \!&\! 0 \!&\! 0 \!&\! 0 \!&\! 0 \!&\! 0 \\
-1 \!&\! 1 \!&\! 0 \!&\! 0 \!&\! 0 \!&\! 0 \!&\! 1 \!&\! 0 \!&\! -1  \!&\! 0  \!&\! 0 \!&\! 0 \!&\! 0 \!&\! 0 \!&\! 0 \!&\! 0 \\
 0 \!&\! 0 \!&\! 0 \!&\! 0 \!&\! 1 \!&\! -1 \!&\! 0 \!&\! 0 \!&\! 0  \!&\! 1  \!&\! -1 \!&\! 0 \!&\! 0 \!&\! 0 \!&\! 0 \!&\! 1 \\
 0 \!&\! 0 \!&\! 0 \!&\! -1 \!&\! 0 \!&\! -1 \!&\! 0 \!&\! 0 \!&\! 0  \!&\! 0  \!&\! 1 \!&\! 0 \!&\! 0 \!&\! 0 \!&\! 0 \!&\! 0 \\
 0 \!&\! 0 \!&\! 0 \!&\! 1 \!&\! 1 \!&\! 0 \!&\! 0 \!&\! -1 \!&\! 0  \!&\! 0  \!&\! 0 \!&\! 0 \!&\! 0 \!&\! 0 \!&\! 0 \!&\! -1 \\
0 \!&\! 0 \!&\! -1 \!&\! 0 \!&\! 0 \!&\! 0 \!&\! 0 \!&\! 1 \!&\! 0  \!&\! 0  \!&\! -1 \!&\! 0 \!&\! 1 \!&\! 0 \!&\! 0 \!&\! 0 \\
1 \!&\! 0 \!&\! 0 \!&\! 0 \!&\! 0 \!&\! 1 \!&\! -1 \!&\! 0 \!&\! 0  \!&\! -1  \!&\! 0 \!&\! 0 \!&\! 0 \!&\! 0 \!&\! 0 \!&\! 0 \\
 0 \!&\! -1 \!&\! 1 \!&\! 0 \!&\! 0 \!&\! 0 \!&\! 0 \!&\! 0 \!&\! 0  \!&\! 1  \!&\! 0 \!&\! -1 \!&\! -1 \!&\! 0 \!&\! 1 \!&\! 0 \\
0 \!&\! 0 \!&\! 0 \!&\! -1 \!&\! 0 \!&\! 0 \!&\! 0 \!&\! 1 \!&\! -1  \!&\! 0  \!&\! 0 \!&\! 1 \!&\! 0 \!&\! 0 \!&\! 0 \!&\! 0 \\
0 \!&\! 0 \!&\! 0 \!&\! 1 \!&\! -1 \!&\! 0 \!&\! 1 \!&\! 0 \!&\! 0  \!&\! 0  \!&\! 0 \!&\! -1 \!&\! -1 \!&\! 1 \!&\! 0 \!&\! 0 \\
0 \!&\! 0 \!&\! 0 \!&\! 0 \!&\! 0 \!&\! 0 \!&\! 0 \!&\! 0 \!&\! 1  \!&\! -1  \!&\! 1 \!&\! 0 \!&\! 0 \!&\! -1 \!&\! 0 \!&\! 0 \\
0 \!&\! 0 \!&\! 0 \!&\! 0 \!&\! 0 \!&\! 0 \!&\! -1 \!&\! 0 \!&\! 1  \!&\! 0  \!&\! 1 \!&\! 0 \!&\! 0 \!&\! 0 \!&\! -1 \!&\! 0
\end{array}\right).
\end{multline*}
 A direct check shows that $\wB_{\text{CG}}^\circ\Omega^\circ=(I\; 0)$, hence the
Cremmer-Gervais bracket is compatible with $\CC_{\text{CG}}^\circ$ by Proposition~\ref{Bomega}.
The exchange matrix $\wB_{\text{CG}}$ for $\CC_{\text{CG}}$ is obtained from  $\wB_{\text{CG}}^\circ$ by deletion of the
rightmost column. The compatibility is verified in the same way as before. 

Assertion (ii) is proved exactly as in the case of $SL_3$, with the help of Proposition~\ref{regfun}; a straightforward computation shows that all assumptions of this Propositions are valid.   

To prove assertion (iii), we parametrize the left and the right action of $\H_{\text{CG}}$ by $\diag(t^3,t,t^{-1},t^{-3})$
and $\diag(z^3,z,z^{-1},z^{-3})$, respectively. Then
 condition 1) of Remark~\ref{tam} holds with $1$-dimensional vectors $\eta_i$, $\zeta_i$ given by
\begin{align*}
\eta_1=1,\quad \eta_2=-1, \quad \eta_3=1, \quad &\eta_4=-3,\quad \eta_5=3, \quad \eta_6=3, \quad 
\eta_7=4,\quad \eta_8=0,\\
\eta_9=2,\quad \eta_{10}=0, \quad \eta_{11}=0, \quad &\eta_{12}=-1,\quad \eta_{13}=1, \quad \eta_{14}=2, \quad 
\eta_{15}=-2,\\
\zeta_1=3,\quad \zeta_2=3, \quad \zeta_3=-3, \quad &\zeta_4=1,\quad \zeta_5=-1, \quad \zeta_6=1, \quad 
\zeta_7=0,\quad \zeta_8=0,\\
\zeta_9=0,\quad \zeta_{10}=4, \quad \zeta_{11}=2, \quad &\zeta_{12}=1,\quad \zeta_{13}=-1, \quad \zeta_{14}=-2, \quad 
\zeta_{15}=2.
\end{align*}
Conditions 2) and 3) are now verified via direct computation.

{\it Case 3}. Here $k_T=2$, and hence the corresponding Belavin-Drinfeld class contains a 1-parameter
family of R-matrices. It is convenient to take the solution to~\eqref{r01},~\eqref{r02} given by
$$
r_0 -\frac12 \mathfrak{t}_0= \frac14 \left( e_{11}\wedge e_{22}-e_{22}\wedge e_{33}-e_{33}\wedge e_{44} + 2 e_{22}\wedge e_{44} - e_{11}\wedge e_{44} \right).
$$
The basis in $\C(SL_4)$ that makes the corresponding bracket diagonal quadratic is given by
\begin{align*}
&P_1= x_{12}, \quad P_2=x_{13},\quad P_3=x_{41},\quad P_4=-x_{42},\quad P_5=-\hat x_{12},\quad P_6=-\hat x_{13},\\
&P_7=\hat x_{41},\quad
P_8=\hat x_{42},\quad
P_9=-\left|\begin{array}{cc} x_{32}& x_{33}\\ x_{42} & x_{43}\end{array}\right|,\quad
P_{10}=\left|\begin{array}{cc} x_{13}& x_{14}\\ x_{43} & x_{44}\end{array}\right|,\\
&P_{11}=-\left|\begin{array}{cc}  x_{12}&  x_{13}\\  x_{42} &  x_{43}\end{array}\right|,\quad
P_{12}=\left|\begin{array}{cc}  x_{13}&  x_{14}\\  x_{23} &  x_{24}\end{array}\right|,\quad
P_{13}=\left|\begin{array}{cc} x_{31}& x_{32}\\ x_{41} & x_{42}\end{array}\right|,\\
&P_{14}=\left|\begin{array}{cc} x_{13}& x_{14}\\ x_{41} & x_{42}\end{array}\right|, \quad
P_{15}=\left|\begin{array}{cc} \hat x_{13}& \hat x_{14}\\ \hat x_{41} & \hat x_{42}\end{array}\right|.
\end{align*}
The coefficient matrix of the bracket in this basis is given by
\begin{multline*}
4\Omega=\\
\left(\begin{array}{rrrrrrrrrrrrrrr}
 0 & -1 & 0  & -3 & -2 & -1 & 0  & 1 & -2  & -4  & 0 & -2 & -2 & -4 & 0 \\
 1 & 0  & -1 & -2 & -1 & 0  & -3 & -2& -4  & 0   & -2 & 2 & -2 & -2 & -2 \\
0  & 1  & 0  & -3 & 0  & -3 & -2 & 1 & -2  & 0   & -2 & 0 & 0 & 2 & -2 \\
3  & 2  & 3  & 0  & 1  & -2 & 1  & 4 & 2   & -2  & 2 & -2 & 2 & 2 & 2 \\
2  & 1  & 0  & -1 & 0  & 1  & 0  & 3 & 0   & -2  & 0 & 2 & 2 & 0 & 4 \\
1  & 0  & 3  & 2  & -1 & 0  & 1  & 2 & 0   & 0   & 2 & -2 & 2 & 2 & 2 \\
0  & 3  & 2  & -1 & 0  & -1 & 0  & 3 & 0   & 2   & 2 & 0 & 0 & 2 & -2 \\
-1 & 2  & -1 & -4 & -3 & -2 & -3 & 0 & -2  & -2  & -2 & 2 & -2 & -2 & -2 \\
2  &  4 & 2  & -2 & 0  & 0  & 0  & 2 & 0   & 0   & 2 & 2 & 2 & 2 & 2 \\
4  & 0  & 0  & 2  & 2  & 0  & -2 & 2 & 0   & 0   & 2 & 2 & 2 & 2 & 2 \\
0  & 2  & 2  & -2 & 0  & -2 & -2 & 2 & -2  & -2  & 0 & 0 & 0 & 0 & 0 \\
2  & -2 & 0  & 2  & -2 &  2 & 0  & -2& -2  & -2  & 0 & 0 & 0 & 0 & 0 \\
2  &  2 & 0  & -2 & -2 & -2 & 0  & 2 & -2  & -2  & 0 & 0 & 0 & 0 & 0 \\
4  &  2 & -2 & -2 & 0  & -2 & -2 & 2 & -2  & -2  & 0 & 0 & 0 & 0 & 0 \\
0  &  2 & 2  & -2 & -4 & -2 & 2  & 2 & -2  & -2  & 0 & 0 & 0 & 0 & 0 
\end{array}\right).
\end{multline*}
A basis in $\C(GL_4)$ is obtained by adding $P_{16}=\det X$ to the above basis. Since $\det X$ is a Casimir function,
the corresponding coefficient matrix $\Omega^\circ$ is obtained from $\Omega$ by adding a zero column on the right 
and a zero row at the bottom.
By assertion (i) of Conjecture~\ref{ulti}, the cluster structure $\CC_{1\mapsto3}$ we are looking for  
should have 4 stable variables; their  images under $\varphi_{1\mapsto3}$ are 
polynomials $P_{12}$, $P_{13}$, $P_{14}$ and $P_{15}$; recall that $P_{16}$ is the image of the fifth stable variable that exists 
in $\CC_{1\mapsto3}^\circ$, but not in $\CC_{1\mapsto3}$. Therefore, the exchange matrix of 
$\CC_{1\mapsto3}^\circ$ is a $11\times 16$ matrix. It is given by
\begin{multline*}
\wB_{1\mapsto3}^\circ=\\
\left(
\begin{array}{rrrrrrrrrrrrrrrr}
 0 \!&\! 1 \!&\! 0 \!&\! 1 \!&\! 0 \!&\! 0 \!&\! 0 \!&\! 0  \!&\! 0  \!&\! 0  \!&\! -1 \!&\! 0 \!&\! 0 \!&\! 0 \!&\! 0 \!&\! 0 \\
-1 \!&\! 0 \!&\! -1 \!&\! 0 \!&\! 0 \!&\! 0 \!&\! 0 \!&\! 0 \!&\! 0  \!&\!-1  \!&\! 1 \!&\! 0 \!&\! 0 \!&\! 1 \!&\! 0 \!&\! 0 \\
0  \!&\! 1 \!&\! 0 \!&\! 1 \!&\! 0 \!&\! 0 \!&\! 0 \!&\! 0 \!&\!  0  \!&\! 0  \!&\! 0 \!&\! 0 \!&\! 0 \!&\! -1 \!&\! 0 \!&\! 0 \\
-1 \!&\! 0 \!&\!-1 \!&\! 0 \!&\! 0 \!&\! 0  \!&\! 0 \!&\! 0 \!&\! -1  \!&\! 0  \!&\! 1 \!&\! 0 \!&\! 1 \!&\! 0 \!&\! 0 \!&\! 0 \\
0 \!&\! 0 \!&\! 0 \!&\!  0 \!&\! 0 \!&\! -1 \!&\! 0 \!&\!-1 \!&\! 0  \!&\! 0  \!&\! 1 \!&\! 0 \!&\! 0 \!&\! 0 \!&\! 0 \!&\! 1 \\
 0 \!&\! 0 \!&\! 0 \!&\! 0 \!&\! 1 \!&\! 0 \!&\! 1 \!&\!  0 \!&\! 0  \!&\! -1 \!&\! 0 \!&\! 1 \!&\! 0 \!&\! 0 \!&\!-1 \!&\! 0 \\
 0 \!&\! 0 \!&\! 0  \!&\! 0 \!&\! 0 \!&\!-1 \!&\! 0 \!&\! -1 \!&\! 0  \!&\! 0  \!&\! 0 \!&\! 0 \!&\! 0 \!&\! 0 \!&\! 1 \!&\! 0 \\
0 \!&\! 0 \!&\! 0 \!&\! 0 \!&\! 1 \!&\! 0 \!&\! 1 \!&\! 0 \!&\! -1  \!&\!  0  \!&\! 0 \!&\! 0 \!&\! 0 \!&\! 0 \!&\! 0 \!&\! -1 \\
 0 \!&\! 0 \!&\! 0 \!&\! 1 \!&\! 0 \!&\! 0 \!&\! 0 \!&\! 1 \!&\! 0  \!&\! 0  \!&\! -1 \!&\! 0 \!&\! -1 \!&\! 0 \!&\! 0 \!&\! 0 \\
 0 \!&\! 1 \!&\! 0 \!&\! 0 \!&\! 0 \!&\! 1 \!&\! 0 \!&\! 0 \!&\! 0  \!&\! 0  \!&\! -1 \!&\! -1 \!&\! 0 \!&\! 0 \!&\! 0 \!&\! 0 \\
1 \!&\! -1 \!&\! 0 \!&\! -1 \!&\! -1 \!&\! 0 \!&\! 0 \!&\! 0 \!&\! 1  \!&\! 1  \!&\! 0 \!&\! 0 \!&\! 0 \!&\! 0 \!&\! 0 \!&\! 0 
\end{array}\right).
\end{multline*}
 A direct check shows that $\wB_{1\mapsto3}^\circ\Omega^\circ=(-I\; 0)$, hence the
 bracket defined above is compatible with $\CC_{1\mapsto3}^\circ$ by Proposition~\ref{Bomega}.
The exchange matrix $\wB_{1\mapsto3}$ for $\CC_{1\mapsto3}$ is obtained from  $\wB_{1\mapsto3}^\circ$ by deletion of the
rightmost column. The compatibility is verified in the same way as before. 

Assertion (ii) is proved exactly as in the previous case. 

To prove assertion (iii), we parametrize the left and the right action of $\H_{1\mapsto3}$ by $\diag(t,w,w^{-1},t^{-1})$
and $\diag(z,u,u^{-1},z^{-1})$, respectively. Then
 condition 1) of Remark~\ref{tam} holds with $2$-dimensional vectors $\eta_i$, $\zeta_i$ given by
\begin{align*}
&\eta_1=(1,0),\quad \eta_2=(1,0),\quad \eta_3=(-1,0), \quad \eta_4=(-1,0),\quad \eta_5=(0,-1), \\
&\eta_6=(0,1), \quad \eta_7=(-1,0),\quad \eta_8=(0,-1),\quad \eta_9=(-1,-1),\quad \eta_{10}=(0,0),\\  
&\eta_{11}=(0,0), \quad \eta_{12}=(1,1),\quad \eta_{13}=(-1,-1), \quad \eta_{14}=(0,0), \quad \eta_{15}=(0,0),\\
&\zeta_1=(0,1),\quad \zeta_2=(0,-1),\quad \zeta_3=(1,0), \quad \zeta_4=(0,1),\quad \zeta_5=(-1,0), \\
&\zeta_6=(-1,0), \quad \zeta_7=(1,0),\quad \zeta_8=(1,0),\quad \zeta_9=(0,0),\quad \zeta_{10}=(-1,-1),\\  
&\zeta_{11}=(0,0), \quad \zeta_{12}=(-1,-1),\quad \zeta_{13}=(1,1), \quad \zeta_{14}=(0,0), \quad \zeta_{15}=(0,0).
\end{align*}
Conditions 2) and 3) are now verified via direct computation.

{\it Case 4}. Here $k_T=2$, and hence the corresponding Belavin-Drinfeld class contains a 1-parameter
family of R-matrices. It is convenient to take the solution to~\eqref{r01},~\eqref{r02} given by
$$
r_0 -\frac12 \mathfrak{t}_0=\frac14 \left(e_{11}\wedge e_{22}+e_{22}\wedge e_{33}+e_{33}\wedge e_{44}-e_{11}\wedge e_{44}\right).
$$
The basis in $\C(SL_4)$ that makes the corresponding bracket diagonal quadratic is given by
\begin{align*}
&P_1= -x_{12}, \qquad P_2=x_{42},\qquad P_3=-x_{41},\\
& P_4=-\hat x_{41},\qquad P_5=\hat x_{42},\qquad P_6=-\hat x_{12},\\
&P_7=x_{12}x_{42}-x_{13}x_{41},\qquad
P_8=\left|\begin{array}{cc} x_{12}& x_{13}\\ x_{42} & x_{43}\end{array}\right|,\\
&P_9=\left|\begin{array}{cc} x_{11}& x_{12}\\ x_{41} & x_{42}\end{array}\right|,\qquad
P_{10}=\left|\begin{array}{cc} x_{12}& x_{13}\\ x_{32} & x_{33}\end{array}\right|,\\
&P_{11}=x_{41}\left|\begin{array}{cc}  x_{13}&  x_{14}\\  x_{33} &  x_{34}\end{array}\right|
-x_{42}\left|\begin{array}{cc}  x_{12}&  x_{14}\\  x_{32} &  x_{34}\end{array}\right|,\\
&P_{12}=x_{14},\qquad
P_{13}=\hat x_{14},\qquad
P_{14}=\left|\begin{array}{cc} x_{21}& x_{23}\\ x_{31} & x_{33}\end{array}\right|, \\
&P_{15}=\hat x_{41}\left(x_{41}\left|\begin{array}{cc} x_{13}& x_{14}\\ x_{23} &  x_{24}\end{array}\right|
-x_{42}\left|\begin{array}{cc} x_{12}& x_{14}\\ x_{22} &  x_{24}\end{array}\right|\right)+\\
&\qquad\qquad\qquad\qquad\hat x_{42}\left(x_{41}\left|\begin{array}{cc} x_{13}& x_{14}\\ x_{33} &  x_{34}\end{array}\right|
-x_{42}\left|\begin{array}{cc} x_{12}& x_{14}\\ x_{32} &  x_{34}\end{array}\right|\right).
\end{align*}
The coefficient matrix of the bracket in this basis is given by
\begin{multline*}
4\Omega=\\
\left(\begin{array}{rrrrrrrrrrrrrrr}
 0 & -3 & 0  & -2 & -1 & -2 & -3  & -2 & 0  & -1 & -3 & -2 & 0 & -2 & -4 \\
 3 & 0  & 3  & -1 & 0  & -1 &  3  & 0  & 2  & 1  &  2 & 1  & 1 &  0 &  2 \\
0  & -3 & 0  & -2 & -1 & -2 &  1  & -2 & 0  & -1 &  1 & 2  & 0 & -2 &  0 \\
2  & 1  & 2  & 0  & 3  &  0 &  3  &  2 & 4  & 1  &  1 &  0 &-2 &  2 &  0 \\
1  & 0  & 1  & -3 & 0  & -3 &  1  &  0 & 2  & -1 & -2 & -1 &-1 &  0 & -2 \\
2  & 1  & 2  & 0  &  3 & 0  &  3  &  2 & 4  & 1  &  1 &  0 & 2 &  2 &  4 \\
3  & -3 & -1 & -3 & -1 & -3 & 0   & -2 & 2  & 0  & -1 & -1 & 1 & -2 & -2 \\
2  & 0  &  2 & -2 &  0 & -2 &  2  & 0  &  4 & 2  &  0 & -2 & 2 &  0 &  0 \\
0  & -2 & 0  & -4 & -2 & -4 & -2  & -4 & 0  & -2 & -2 &  0 & 0 & -4 & -4 \\
1  & -1 & 1  & -1 & 1  & -1 &   0 & -2 & 2  & 0  & -3 & -3 &-1 &  2 & -2 \\
3  &-2  & -1 & -1 & 2  & -1 &  1  &  0 &  2 & 3  & 0  &  1 & 1 &  0 &  2 \\
2  & -1 & -2 &  0 & 1  & 0  &  1  &  2 & 0  & 3  & -1 & 0  & 2 & -2 &  0 \\
0  & -1 & 0  &  2 & 1  & -2 & -1  & -2 & 0  & 1  & -1 & -2 & 0 & 2  & 0 \\
2  &  0 &  2 & -2 & 0  & -2 &  2  &  0 & 4  & -2 & 0  & 2  & 0 &  0 & 0 \\
4  & -2 & 0  &  0 &  2 & -4 & 2   &  0 & 4  &  2 & -2 & 0  & 0 & 0  & 0 
\end{array}\right).
\end{multline*}
A basis in $\C(GL_4)$ is obtained by adding $P_{16}=\det X$ to the above basis. Since $\det X$ is a Casimir function,
the corresponding coefficient matrix $\Omega^\circ$ is obtained from $\Omega$ by adding a zero column on the right 
and a zero row at the bottom.
By assertion (i) of Conjecture~\ref{ulti}, the cluster structure $\CC_{1\mapsto2}$ we are looking for  
should have 4 stable variables; their  images under $\varphi_{1\mapsto2}$ are 
polynomials $P_{12}$, $P_{13}$, $P_{14}$ and $P_{15}$; recall that $P_{16}$ is the image of the fifth stable variable that exists 
in $\CC_{1\mapsto2}^\circ$, but not in $\CC_{1\mapsto2}$. Therefore, the exchange matrix of 
$\CC_{1\mapsto2}^\circ$ is a $11\times 16$ matrix. It is given by
\begin{multline*}
\wB_{1\mapsto2}^\circ=\\
\left(
\begin{array}{rrrrrrrrrrrrrrrr}
 0 \!&\! 1 \!&\! -1 \!&\! 0 \!&\! 0 \!&\! 0 \!&\! 1 \!&\! -1  \!&\! 0  \!&\! 0  \!&\! 0 \!&\! 0 \!&\! 0 \!&\! 0 \!&\! 0 \!&\! 0 \\
-1 \!&\! 0 \!&\! -1 \!&\! 0 \!&\! 0 \!&\! 0 \!&\! 0 \!&\! 0 \!&\! 1  \!&\! 0  \!&\! 0 \!&\! 0 \!&\! 0 \!&\! 0 \!&\! 0 \!&\! 0 \\
1  \!&\! 1 \!&\! 0 \!&\! 0 \!&\! 0 \!&\! 0 \!&\! -1 \!&\! 0 \!&\!  0  \!&\! 0  \!&\! 0 \!&\! 0 \!&\! 0 \!&\! 0 \!&\! 0 \!&\! 0 \\
0 \!&\! 0 \!&\! 0 \!&\! 0 \!&\! -1 \!&\! 0  \!&\! 0 \!&\! 0 \!&\! 0  \!&\! 0  \!&\! -1 \!&\! 0 \!&\! 0 \!&\! 0 \!&\! 1 \!&\! 0 \\
0 \!&\! 0 \!&\! 0 \!&\!  1 \!&\! 0 \!&\! 1 \!&\! 0 \!&\! -1 \!&\! 1  \!&\! 0  \!&\! 0 \!&\! 0 \!&\! 0 \!&\! -1 \!&\! 0 \!&\! -1 \\
 0 \!&\! 0 \!&\! 0 \!&\! 0 \!&\! -1 \!&\! 0 \!&\! 0 \!&\!  0 \!&\! 0  \!&\! 1 \!&\! 0 \!&\! 0 \!&\! -1 \!&\! 0 \!&\! 0 \!&\! 1 \\ 
-1 \!&\! 0 \!&\! 1  \!&\! 0 \!&\! 0 \!&\! 0 \!&\! 0 \!&\! 0 \!&\! 0  \!&\! 1  \!&\! -1 \!&\! 1 \!&\! 0 \!&\! 0 \!&\! 0 \!&\! 0 \\
1 \!&\! 0 \!&\! 0 \!&\! 0 \!&\! 1 \!&\! 0 \!&\! 0 \!&\! 0 \!&\! -1  \!&\! -1  \!&\! 0 \!&\! 0 \!&\! 0 \!&\! 0 \!&\! 0 \!&\!  0 \\
 0 \!&\! -1 \!&\! 0 \!&\! 0 \!&\! -1 \!&\! 0 \!&\! 0 \!&\! 1 \!&\! 0  \!&\! 0  \!&\! 0 \!&\! 0 \!&\! 0 \!&\! 1 \!&\! 0 \!&\! 0 \\
 0 \!&\! 0 \!&\! 0 \!&\! 0 \!&\! 0 \!&\! -1 \!&\! -1 \!&\! 1 \!&\! 0  \!&\! 0  \!&\! 1 \!&\! 0 \!&\! 0 \!&\! 0 \!&\! 0 \!&\! 0 \\
0 \!&\!  0 \!&\! 0 \!&\! 1 \!&\! 0 \!&\! 0 \!&\! 1  \!&\! 0 \!&\! 0  \!&\! -1  \!&\! 0 \!&\! 0 \!&\! 1 \!&\! 0 \!&\! -1 \!&\! 0 
\end{array}\right).
\end{multline*}
 A direct check shows that $\wB_{1\mapsto2}^\circ\Omega^\circ=(I\; 0)$, hence the
bracket defined above is compatible with $\CC_{1\mapsto2}^\circ$ by Proposition~\ref{Bomega}.
The exchange matrix $\wB_{1\mapsto2}$ for $\CC_{1\mapsto2}$ is obtained from  $\wB_{1\mapsto2}^\circ$ by deletion of the
rightmost column. The compatibility is verified in the same way as before. 

Assertion (ii) is proved exactly as in the previous case. 

To prove assertion (iii), we parametrize the left and the right action of $\H_{1\mapsto2}$ by $\diag(t,w,t^{-1}w^2,w^{-3})$
and $\diag(z,u,z^{-1}u^2,u^{-3})$, respectively. Then
 condition 1) of Remark~\ref{tam} holds with $2$-dimensional vectors $\eta_i$, $\zeta_i$ given by
\begin{align*}
&\eta_1=(1,0),\quad \eta_2=(0,-3),\quad \eta_3=(0,-3), \quad \eta_4=(-1,0),\quad \eta_5=(0,-1), \\
&\eta_6=(0,-1), \quad \eta_7=(1,-3),\quad \eta_8=(1,-3),\quad \eta_9=(1,-3),\quad \eta_{10}=(0,2),\\  
&\eta_{11}=(0,-1), \quad \eta_{12}=(1,0),\quad \eta_{13}=(0,3), \quad \eta_{14}=(-1,-1), \quad \eta_{15}=(0,-2),\\
&\zeta_1=(0,1),\quad \zeta_2=(0,1),\quad \zeta_3=(1,0), \quad \zeta_4=(0,3),\quad \zeta_5=(0,3), \\
&\zeta_6=(-1,0), \quad \zeta_7=(0,2),\quad \zeta_8=(-1,3),\quad \zeta_9=(1,1),\quad \zeta_{10}=(-1,3),\\  
&\zeta_{11}=(0,-1), \quad \zeta_{12}=(0,-3),\quad \zeta_{13}=(-1,0), \quad \zeta_{14}=(1,1), \quad \zeta_{15}=(0,2).
\end{align*}
Conditions 2) and 3) are now verified via direct computation.
\end{proof}

\section{The case of triangular Lie bialgebras}\label{SecTri}
We conclude with an example that shows that Conjecture \ref{ulti} is not
valid in the case
of skew-symmetric R-matrices, that is, R-matrices $r$ that satisfy
(\ref{rBD}) together with the condition
$$
r + r^{21}=0.
$$
Consider the simplest skew-symmetric R-matrix in $sl_2$:
$$
r=\left (\begin{array}{cc} 1& 0\\0 &-1\end{array}\right) \wedge \left
(\begin{array}{cc} 0& 1\\0 & 0\end{array}\right).
$$
Let $X=(x_{ij})_{i,j=1}^2$ denote an element of $SL_2$. Choose
functions
$y_1=x_{11}$, $y_2=x_{21}$, $y_3=x_{11}-x_{22}$ as coordinates on $SL_2$. Then a
direct calculation
using (\ref{sklya}) shows that the Poisson-Lie bracket corresponding to $r$
has the form
$$
\{y_1, y_2\} = y_2^2, \qquad \{y_1, y_3\} = y_2 y_3, \qquad \{y_2, y_3\} = 0.
$$
Select a new coordinate system on the open dense set $\{x_{21} \ne 0\}$:
$z_1=y_1$,  $z_2=-1/y_2$, $z_3= y_3/y_2$, then the Poisson algebra above
becomes:
$$
\{z_1, z_2\} = 1, \qquad \{z_1, z_3\} = \{z_2, z_3\} = 0.
$$
It is easy to see that both collections $z_1, z_2, z_3$ and $y_1, y_2, y_3$
generate
the set of rational functions on $SL_2$. However, we claim that there is no
triple
of independent rational functions $p_i(z_1,z_2,z_3)$ such that
$\{p_i,p_j\} = c_{ij} p_ip_j$ for some constants $c_{ij}$, $i,j=1,2,3$.
Indeed, the independence implies that at least one of the constants
$c_{ij}$, say, $c_{12}$,
is nonzero. View $p_1$ and $p_2$  as ratios of two polynomials in $z_1$
with the difference of degrees of the numerator and denominator equal to
$\delta_1$ and $\delta_2$, respectively. Then the difference
of degrees of the numerator and denominator of $\{p_1,p_2\} $ viewed as a
rational function of $z_1$ is at most $\delta_1 + \delta_2 -1$, and thus
$\{p_1,p_2\} $ cannot be a nonzero multiple of $p_1p_2$.
This means, in particular, that the Poisson structure associated with the
R-matrix above cannot be compatible with any cluster structure in the field
of rational functions on $SL_2$.

\section*{Acknowledgments}

M.~G.~was supported in part by NSF Grant DMS \#0801204. 
M.~S.~was supported in part by NSF Grants DMS \#0800671 and PHY \#0555346.  
A.~V.~was supported in part by ISF Grant \#1032/08. This paper was partially written during
the joint stay of all authors at MFO Oberwolfach in August 2010 within the framework
of Research in Pairs programme and the visit of the third author to  Institut des Hautes
\'Etudes Scientifiques in September--October 2010. We are grateful to these institutions for
warm hospitality and excellent working conditions.

 \end{document}